\documentclass{amsart}
\usepackage{fullpage, amsmath, amssymb, amsthm, enumerate, url, bibentry, hyperref, color,xcolor}
\usepackage[all,cmtip]{xy}
\usepackage{tikz,epsfig}
\usepackage{verbatim}
\usepackage{geometry}
\usepackage{dirtytalk}
\usepackage{float}
\usepackage{colonequals}
\usepackage{placeins}
\usepackage{cite}
\usepackage{tabularray}
\usepackage{cleveref}
\usepackage{caption}
\usepackage{arydshln}
\usepackage{mathscinet}

\usepackage{tikz-cd}

\definecolor{purple}{rgb}{0.59, 0.44, 0.84}

\definecolor{orange}{rgb}{1, 0.6, 0.4}

\newtheorem{theorem}{Theorem}[section]
\newtheorem{lemma}[theorem]{Lemma}
\newtheorem{corollary}[theorem]{Corollary}
\newtheorem{proposition}[theorem]{Proposition}

\newtheorem{defn}{Definition}[section]

\newcounter{example}

\newcounter{remark}
\newenvironment{remark}[1][]{\refstepcounter{remark}\par\medskip
   \noindent \textbf{Remark~\theremark. #1} \rmfamily}{\medskip}

\numberwithin{equation}{section}
\numberwithin{example}{section}
\numberwithin{remark}{section}

\def \l {\lambda}

\def\C{\mathbb C}

\def\Z{{\mathbb Z}}

\def\Q{{\mathbb Q}}

\def\C{\mathbb{C}}

\def\Z{\mathbb{Z}}
\def\Q{\mathbb{Q}}

\def \ol{\overline}

\newcommand*\HYPERskip{&}
\catcode`,\active
\newcommand*\pFq{
\begingroup
\catcode`\,\active
\def ,{\HYPERskip}%
\doHyper
}
\catcode`\,12
\def\doHyper#1#2#3#4#5{%
\, _{#1}F_{#2}\left[\begin{matrix}#3 \smallskip \\  #4\end{matrix} \; ; \; #5\right]%
\endgroup
}

\def \pfq{\pFq}

\catcode`,\active

\catcode`\,12

\crefformat{section}{\S#2#1#3}
\crefformat{subsection}{\S#2#1#3}

\catcode`,\active

\catcode`\,12

\catcode`,\active

\catcode`\,12

\def\C{\mathbb{C}}

\def\Z{\mathbb{Z}}
\def\Q{\mathbb{Q}}

\def \ol{\overline}

\title{Modular Forms and Certain ${}_2F_1(1)$ Hypergeometric Series}
\author{Esme Rosen }

\begin{document}
\begin{abstract}
   Using the framework relating hypergeometric motives to modular forms, we define an explicit family of weight 2 Hecke eigenforms with complex multiplication. We use the theory of ${}_2F_1(1)$ hypergeometric series and Ramanujan's theory of alternative bases to compute the exact central $L$-value of these Hecke eigenforms in terms of special beta values. We also show the integral Fourier coefficients can be written in terms of Jacobi sums, reflecting a motivic relation between the hypergeometric series and the modular forms.
\end{abstract}
\maketitle

\section{Introduction}
In \cite{HMM2}, Allen, Grove, Long, and Tu in the course of developing the Explicit Hypergeometric-Modularity method use the integral formula for hypergeometric series and Ramanujan's theory of alternative bases to relate special $L$-values of modular forms to a ${}_3F_2$ hypergeometric function evaluated at 1.  This work on exact $L$-values in terms of hypergeometric series is inspired by other formulas---see e.g. \cite{lilongtu}, \cite{osburnstraub}, \cite{zagierarith}, and \cite{zudlin}. However, the methods developed in \cite{HMM2}, \cite{rosen}, and this paper apply more uniformly for a larger class of examples. In this note, we will explain how their method works using certain ${}_2F_1(1)$. Since these reduce to beta values $B(\alpha,\beta)$ for rational $\alpha$ and $\beta$, the underlying motives are Jacobi motives (see \cite{otsubo}), which are much better understood than motives studied in \cite{HMM2}. The first half of this paper constructs a family of Hecke eigenforms with complex multiplication (CM) as linear combinations of eta products and shows that they have nice properties due to the underlying Jacobi motive. Recall that a modular form has CM by $-D$ if the Fourier coefficients $a_n(f)$ satisfy $a_n(f)=a_n(f)\cdot \chi_{-D}(n)$, where  $\chi_{-D}(n):=\left (\frac {-D}{n}\right)$ is the Jacobi symbol. Once we have constructed these Hecke eigenforms we provide a formula for their Fourier coefficients, and we compute their exact $L$-values at 1. All of our $L$-values can be expressed in terms of the periods of elliptic curves with CM by an imaginary quadratic field with discriminant $-D$, which are called \textit{Chowla-Selberg periods} and can be written explicitly as $$
  \Omega_{-D}:=\sqrt{\pi}\left(\prod_{i=1}^{D-1}\Gamma(i/D)^{\chi_{-D}(i)}\right)^{1/2h'(-D)},
$$ (see e.g. \cite{grossror}). Here $\Gamma(x)$ is the gamma function, and $h'(-D)$ is the class number $h(D)$ of the field $\Q(\sqrt{-D})$, divided by the numbers of roots of unity in that field. These periods have a variety of applications, especially in transcendence results, where they have been used to show $\Gamma(1/d)$ for $d=3,4,6$ is transcendental. The transcendence of $\Gamma(1/d)$ for any other positive integers $d$ remains an open problem.

We work with hypergeometric parameters $\{\{1/2,r\},\{1,s\}\}$ where $(r,s)$ lies in the finite set  \begin{equation}\label{S1}\mathbb S_1=\{(r,s)\in\Q^2\mid 0<r<s<3/2,\, 1/2<s-r, 24s\in\Z, 8(r+s)\in\Z.\}\end{equation} Our main result can be summarized as follows.
\begin{theorem}
  Let $M:=\text{lcd}(r,s,1/2)$ be the least common denominator of $r$, $s$, and $1/2$, and $G_M$ be the absolute Galois group of $\Q(\zeta_M)$. Further, take $\Omega_{-D}$ to be the Chowla-Selberg period of the imaginary quadratic field $\Q(\sqrt{-D})$.  Then for each pair $(r,s)\in \mathbb{S}_1,$ there is an associated Gr\"ossencharakter $\chi_{r,s}$ and an explicit Hecke eigenform $f_{r,s}$ with complex multiplication by $\Q(\sqrt{-D})$, so that the Deligne representation attached to $f_{r,s}$ is a sub-representation of the induction $\text{Ind}_{G_{M}}^{G_{\Q}}\chi_{r,s}$.  Moreover, for an explicit algebraic number $\alpha_{r,s}$, $$L(f_{r,s},1)=\alpha_{r,s}\cdot B(r,s-r-1/2)\in \Omega_{-D}\cdot \bar{\Q}.$$  A list of exact $L$-values is provided in Section \ref{lvalues}.\end{theorem}

   The Hecke eigenforms $f_{r,s}$ can be written as a linear combination of eta products. The eta products which are Hecke eigenforms have been classified by Martin in \cite{martin}. In particular, there are exactly 12 eta products which are Hecke eigenforms of weight 2, listed in Martin and Ono \cite{onomartin}. Five of these twelve have complex multiplication. Of these five, four are obtained using our methods, labeled 4.1, 4.2, 3.1, and 3.2 in Table \ref{tab:K1}. In this sense, the hypergeometric modularity method used to construct $f_{r,s}$ is an extension of eta products as in \cite{onomartin} to linear combinations of eta products. The modularity results are already known from a theoretical point of view---see e.g. Darmon \cite{darmon} and Ribet \cite{ribet}---but the advantage of our approach is the determination of the precise modular form and the direct formula involving Jacobi sums in Theorem \ref{gal}.

 The $L$-values are completely explicit, and are useful in the recent work of Allen, Grove, Long, and Tu \cite{HMM2} on well-poised ${}_4F_3(-1)$ hypergeometric series as certain products of $L$-values. The exact $L$-values are also of interest in the Birch and Swinnerton-Dyer (BSD) conjecture, where the values are conjectured to be related to the geometric invariants of the underlying Abelian variety. The weak form of the conjecture states that for an elliptic curve $E/\Q$, the order of the zero at $L(E,1)$ is equal to the rank of the elliptic curve $E$. There is also a generalization of this conjecture to Abelian varieties. For this reason, these $L$-values have been computed numerically with a high degree of precision. See for instance \cite{dabwad} for one such approximation. We have verified all of our $L$-values numerically with the $L$-functions and Modular Forms Database (LMFDB). However, the exact $L$-values for rank 0 curves are also of interest due to another conjectured formula relating the $L$-value to several invariants of the elliptic curve, including most notably the order of the Tate-Shaferevich group. For instance, example 4.1 in the tables of Section \ref{lvalues} says that  $$L(\eta(4\tau)^2\eta(8\tau)^2,1)=\Omega_{-4}.$$ Based on the tables in Martin and Ono \cite{onomartin}, this modular form corresponds to an elliptic curve of conductor 32, for example the curve with affine equation $y^2=x^3+4x$. The weak BSD conjecture tells us that this elliptic curve has rank zero. For rank zero curves, the weak BSD conjecture is a theorem due to Gross and Zagier---see \cite{bsd} for a survey of known results. Using the general formula (see, e.g., \cite{bsd} for the precise formula), we determine the (conjectured) order of the Tate--Shafarevich group is 1. Note that the exact $L$-value at 1 for this particular elliptic curve is already known due to the work of Tunnell \cite{tunnel}. This illustrates that for elliptic curves over $\Q$, there is an extensive literature regarding special $L$-values because of this conjecture besides those cited. However, our modular forms are not always associated with an elliptic curve---in some cases, the associated geometric object is an Abelian variety of $GL_2$-type over $\Q$. For these cases, the exact $L$-values computed are likely all new.

The work in this paper is a direct application of the hypergeometric modularity method, which is introduced for ${}_3F_2(1)$ series in \cite{HMM1}, and studied further in \cite{HMM2}, and \cite{rosen}. Compared to those papers, the geometric setting here is much simpler, primarily for two reasons: 1) the hypergeometric series we work with are periods on curves, and 2) these curves are quotients of Fermat curves, and so the Jacobian of this curve has complex multiplication. For comparison, the base varieties for the ${}_3F_2(1)$ are hypergeometric surfaces, which are less studied in the literature than Fermat curves, and these surfaces only have a CM structure in special cases. In particular, the theoretical construction of the Hecke eigenforms $f_{r,s}$ is much cleaner than in \cite{rosen} due to the modular parametrization of Fermat curves, as studied in Rohrlich's thesis \cite{rohrlich}. 

In \cite{HMM1}, the authors prove modularity results similar to our main modularity theorem using supercongruences. In our setting, we do not have these supercongruences. However, as a byproduct of 1), our modular forms are of weight 2 rather than weight 3, and so an \say{ordinary congruence}, i.e., a mod $p$ congruence, is sufficient to prove our results since the Deligne bound is stronger for lower weight. The $L$-values are computed in the same way as in \cite{HMM2} and \cite{rosen}. However, because of the CM structure, the transcendence properties studied \cite{rosen} are already known in this setting. As a result, we focus on computing the $L$-values themselves and briefly observe that sometimes the $L$-series differ by a twist using a similar method to \cite{rosen}.

\subsection*{Acknowledgments} The work on this paper was conducted concurrently with \cite{HMM2}. As such, the author is grateful to the authors of that paper, especially Ling Long and Fang-Ting Tu, for helpful discussions on the subject. She would also like to thank from Michael Allen, Gene Kopp, Ken Ono, Hasan Saad for valuable feedback on improving the presentation of this paper.

The author is supported by a summer research assistantship from the Louisiana State University Department of Mathematics.

\section{Preliminaries}
The classical hypergeometric series is defined as $$\pfq{2}{1}{a,b}{,c}{\lambda}=\sum_{n=0}^\infty \frac{(a)_n(b)_n}{(c)_n}\frac{\lambda}{n!},$$ where $(a)_n=x(x+1)...(x+n-1)$ is the rising factorial.
Assume $B(a,b)=\Gamma(a)\Gamma(b)/\Gamma(a+b)$ is the usual beta function. Recall the integral formula for ${}_2F_1(1)$ hypergeometric series evaluated at 1---see e.g.,\cite{aar}.---states that
\begin{align*}
   B(r,s-r) \pfq{2}{1}{1/2&r}{&q} {1}&
     =\int_0^1 \lambda^{r}(1-\lambda)^{s-r-3/2}\frac{d\lambda}{\lambda}.
\end{align*} Following the hypergeometric modularity method of \cite{HMM1}, we define the functions
  \begin{align*}
             \mathbb{K}_1(r,s)(\tau)&=2^{1-4r}\lambda^{r}(1-\lambda)^{s-r-3/2}\frac{d\lambda}{\lambda d\tau}\\&=2^{1-4r}\lambda^r(1-\lambda)^{s-r-1/2}\theta_3^4(\tau).
  \end{align*}
  Here we parametrize the variable $\lambda$ by the modular lambda function using Ramanujan's theory of alternative bases, and the relation with the Jacobi theta function $\theta_3$ is given in \cite{piandagm}. This is a weakly holomorphic weight 2 modular form. Similar to the weight 3 case, the integral formula implies the $L$-value is related to the hypergeometric series. Recall that $$L(f,1)=\int_0^{i\infty}f(\tau)d\tau.$$ Take $N={48}/{\text{gcd}(24r,24)}$ as in the weight 3 case. Then $$L(r,s):=  L(\mathbb{K}_1(r,s)(N\tau),1)=\frac{2^{1-4r}B(r,s-r)}{N}\pFq{2}{1}{1/2&r}{&s}{1}.$$  Using the Gauss evaluation formula (see e.g. \cite{aar}), this reduces to 
    \begin{equation}\label{eq:K1-L}     
     L(r,s)=\frac{2^{1-4r}}{N}B(r,s-r-1/2).
    \end{equation}
The value $L(r,s)$ is also a period on a hypergeometric variety introduced by Wolfart \cite{wolfart} and discussed in greater depth by Archinard \cite{archinard}, defined by the desingularization of the affine equation \begin{equation}\label{crs}
   C(r,s):\quad y^M=x^{M(1-r)}(1-x)^{M(3/2+r-s) }.
\end{equation}

 Let $X(r,s)$ denotes the projective desingularization of $C(r,s)$ given explicitly by Archinard \cite{archinard}. She also describes a subvariety of the Jacobian $\text{Jac}(X(r,s))$ which corresponds to \say{new} holomorphic differential 1-forms, a $\Q$-simple piece of the decomposition of $\text{Jac}(X(r,s))$. The subvariety is called the new Jacobian, and is denoted by $J_{new}(r,s)$. The dimension of $J_{new}(r,s)$ is $\varphi(M)/2$. To understand why $J_{new}(r,s)$ corresponds to new differential forms, we write down a basis for $H^0(X(r,s))$, the vector space of differential 1-forms on $X(r,s)$. One such  basis is \begin{equation}\label{omegan}
    \omega_n:=\frac{(1-x)^ax^b}{y^n},
\end{equation} where for $a$, $b$, and $n$ integers less than $M-1$ satisfying regularity conditions given by Archinard \cite{archinard}. If $k
\mid M$, $\omega_k$  can be defined on a subvariety of $X(r,s)$, and so we only want to consider cases where $(n,M)=1$. These are the differentials on $J_{new}(r,s)$. 

 There is a notion of conjugates which is defined in exact the same way as in \cite{HMM1} and \cite{rosen}. Geometrically, this definition tells us that if $\mathbb{K}_1(r,s)$ and $\mathbb{K}_1(r',s')$ are conjugate, they can be defined on the same variety $J_{new}(r,s)$. We recall the formal condition here described in \cite{HMM1} below.  \begin{defn}\label{conj}
   Two pairs $(r,s)$ and $(r',s')$ are \textit{conjugate}  if there exists an integer $c$ coprime to $6$ so that $r-cr'$ and $s-cs'$ are both integers. 
 \end{defn} Note that we can define conjugates that do not belong to $\mathbb{S}_1$. The non-holomorphic conjugates are still defined on $J_{new}(r,s)$, and so we will use this later when discussing differentials of the second kind. 

Similar to the weight 3 cases in \cite{HMM1},  $\mathbb{K}_1(r,s)$ can be expressed as an eta product. This is because $\theta_3$ and $\lambda$ already have known eta product formulas, given in \cite{piandagm}; and see also Lemma 3.1 of \cite{HMM1}.

\begin{proposition}\label{prel}
  We have
   $$
     \mathbb{K}_1(r,s)(\tau)=\frac{\eta(\tau/2)^{16s-8r-16}\eta(2\tau)^{8r+8s-12}}{\eta(\tau)^{24s-32}}.
    $$ 
  Moreover, $\mathbb{K}_1(r,s)(\tau)$ is a congruence cusp form if the exponents are integral, and is holomorphic if  $(r,s)\in \mathbb S_1$ as defined in \eqref{S1}.  As a congruence modular form, the level of $\mathbb{K}_1(r,s)(\tau)$  is $\mathcal{N}:=48^2/(\gcd(24r,24)\gcd(24(s-r-1/2),24)).$ 

\end{proposition}
The proof of this proposition is very similar to the $\mathbb K_2(r,s)$ case in \cite{HMM1} and will be omitted here.

 \section{Hecke Operators}
Similar to the weight 3 case studied in \cite{rosen}, the Hecke operators play an important role in understanding the $\mathbb{K}_1(r,s)$ functions. In this section, we first prove that the set of conjugates of $\mathbb{K}_1(r,s)$ are fixed by the Hecke operators for any $(r,s)\in \mathbb{S}_1$. That is, using the definition given in \cite{rosen}, every family is Galois. This is different from the weight 3 case, for which there are non-Galois families. The reason for this difference is the underlying hypergeometric varieties on which $\mathbb{K}_1(r,s)$ is defined are much nicer. In the first part, we provide some background on Jacobi motives. In the second part, we use this prove the Hecke operators fix each family.

The main theorem of the section is then the following.
\begin{theorem}\label{cons}
    Assume $(r_1,s_1)\in \mathbb S_1$ and its holomorphic conjugates are given by $(r_i,s_i)$. Then each of the $\mathbb{K}_1(r_i,s_i)$ functions lie in the same Hecke orbit of some weight 2 modular form $f_{r_1,s_1}$ with CM by a fundamental discriminant dividing 24. Moreover, there are algebraic constants $\beta_i$ so that $$f_{r_1,s_1}=\beta_0\mathbb{K}_1(r_1,q_1)+\beta_1\mathbb{K}_1(r_2,q_2)+....+\beta_{n-1}\mathbb{K}_2(r_n,q_n).$$
\end{theorem}
This is the analogue of Theorem 3.5 in \cite{rosen}. This Theorem allows us to study exact $L$-values of the Hecke eigenform as well when we integrate both sides.
\begin{remark}
    We will assume that function $f_{r_i,s_i}$ does not depend on the choice of $i$. To resolve the ambiguity introduced by the quadratic twists, we fix a choice of $f_{r_i,s_i}$ for each family in Table \ref{tab:K1}. Therefore, we will refer to $f_{r,s}$ as mentioned in the introduction as simply $f$ throughout this section when the choice of conjugate family is clear.
\end{remark}

 \subsection{Jacobi Motives}
 Jacobi motives originated as curves whose point counts were given in terms of Jacobi sums, discussed later in Section \ref{jac}. Moreover, Jacobi motives are precisely the motives on which the beta function is a period. This was first shown by Rohrlich in the appendix for \cite{grossror}. The base varieties for these motives are quotients of the Jacobians of a \textit{Fermat curve}, which in affine coordinates affine coordinates are $$F_n: x^n+y^n=1.$$ These quotients have the form $$F(R,S): y^M=x^R(1-x)^S,$$ where $Q$ and $R$ are positive integers, and $R+S\leq M$. See for instance \cite{irakawasasaki}. 
 
 The classical work of Rohrlich \cite{rohrlich}, Koblitz and Rohrlich \cite{kobror}, Noburu Aoki \cite{aoki}, and many others has shown that the Jacobian variety $\text{Jac}(F_n)$ decomposes into a product of $\Q$-simple Abelian varieties up to isogeny, and determined precisely when each factor is absolutely simple, i.e. simple over $\bar{\Q}$. In these papers, the classification is given in terms of triple of integers in $(\Z/M\Z)\times$, $(Q,R,S)$, that satisfy $Q+R+S=M$. The relation between Archinard's varieties $J_{new}(r,s)$ is very straightforward. 

\begin{lemma}\label{aok}
 Let $R(r,s)=R=Mr$ and  $S(r,s)=S=M(s-r-1/2)$ and assume $T$ is equivalent to $M-R-S$ modulo $M$. Then the Abelian variety $J_{new}(1-r,3-s)$ is isogenous to the factor of the Jacobian labeled $(Q,R,S)$ in the decomposition of the Jacobian for the Fermat curve $F_M$. 
\end{lemma}

\begin{proof}
 Note that $F(R,S)$ has the same shape as the variety $C(r,s)$ defined in \eqref{crs}, except that we have that $M(1-r)+M(3/2+r-s)=M(5/2-s)$, which is always greater than $M$ by the definition of $\mathbb S_1$. To evade this problem, we the base variety for a non-holomorphic conjugate of $\mathbb{K}_1(r,s)$, specifically $\mathbb{K}_1(1-r,3-s)$. Here, $\mathbb{K}_1(r,s)$ is still a period of a differential of the first kind on $C(1-r,3-s)$, just that differential is not necessarily $dx/y$ as in Archinard's construction. Note that $$C(1-r,3-s): y^M= x^{Mr}(1-x)^{M(s-r-1/2)}$$ satisfies $Mr+M(s-r-1/2)=M(s-1/2)\leq M$.  Therefore, from the point of view of Fermat curves, this is the \say{correct} base variety to use for the family of conjugates associated to $(r,s)$.  By the work of, e.g. \cite{archinard}, the Jacobian of $J_{new}(1-r,3-s)$ is $\Q$-simple, and the $\mathbb{K}_1(r,s)$ are differentials on this variety. Moreover, the family of holomorphic conjugates generates the space of differentials of the first kind. This allows us to compute the period matrix for $J_{new}(1-r,3-s)$, which is equal to the period matrix for the triple $(R,S,T)$ up to a constant multiple in Koblitz and Rohrlich \cite{kobror}, and therefore the two complex tori are isogenous.
\end{proof}

  \noindent\textit{Notation}:  For convenience, we will use the notation $J_{new}(R,S)$ for $J_{new}(1-r,3-s)$ from now on.
  \bigskip
  
We then appeal to methods of Rohrlich \cite{rohrlich}, who realizes a Fermat curve as a modular curve, which implies $J_{new}(R,S)$ is a factor of the Jacobian for some modular curve.

\begin{lemma}\label{sub}
    There exists a subgroup $\Gamma$ of $SL_2(\Z)$ of finite index so that $X(\Gamma)\cong X(r,s)$.
\end{lemma}
\begin{proof}
   The proof is identical to Rohrlich's proof that Fermat curves can be realized as modular curves for a finite index subgroup of $SL_2(\Z)$.
\end{proof}

\begin{remark}
       A related Theorem of Ribet \cite{ribet} states that a simple Abelian variety over $\Q$ of $GL_2$-type is a factor of factor of $J_1(N)$ up to isogeny. In general, quotients of Fermat curves may not be of $GL_2$-type, however. 
    \end{remark}

\subsection{The Action of the Hecke Operators}
\begin{theorem}
    For any conjugate $\{\mathbb{K}_1(r_i,s_i)\}$ family contained in $\mathbb{S}_1$, each $\mathbb{K}_1(r_i,s_i)$ lies in the orbit of the same Hecke eigenform with complex multiplication. Moreover, if $V$ is the $\C$-vector space generated by the set above, $T_pV=V$ for all $p$ coprime to $6$.
\end{theorem}
\begin{proof} 
    We rely on Lemma \ref{sub}, which says $X(r,s)$ is isomorphic to a modular curve $X(\Gamma)$. This implies the differentials of $X(r,s)$ correspond to weight 2 modular forms on $\Gamma$. If $\Gamma(N)\subset\Gamma$ for some $N$, i.e. $\Gamma$ is congruence, then $M_2(\Gamma)\subset M_2(\Gamma(N))$.  As $\Gamma(N)\subset\Gamma$, there is a projection $\pi:X(N)\to X(\Gamma)$, which induces a map of the Jacobians as well. Recall $\Gamma(N)$ is conjugate to a group containing $\Gamma_1(N^2)$, so we can understand the Jacobian of $X(\Gamma)$ in terms of the Jacobian of $\Gamma_1(N^2)$. The latter is well-studied, and is denoted by $J_1(N^2)$. A useful property of $J_1(N^2)$ is that it decomposes up to isogeny into a product of $$J_1(N^2)=A_1^{e_1}\times A_2^{e_2}\times...\times A_{m}^{e_m},$$ where $A_i$ are pairwise non-isogenous Abelian varieties that are $\Q$-simple and are associated to a twist class of newforms on a subgroup of $\Gamma_1(N^2)$, and the $e_i$ indicates the level for the newform (see e.g. \cite{diamondshurman}). By Lemma \ref{aok}, up to isogeny $J_{new}(R,S)$ is isomorphic to $\Q$-simple factor of $\text{Jac}(F_M)$, which we denote by $A_{R,S}$. 
    By taking the pullback of the projection, we can associate $A_{R,S}$ to $A_i$ for some $i$. This implies $A_{R,S}\cong A_f$ for some newform $f$, where $A_f$ is the Abelian variety attached to $f$. Moreover, because any factor $A_{R,S}$ of a Fermat curve has generalized complex multiplication, $f$ must be CM as well. As shown in e.g. \cite{diamondshurman}, the Hecke operators stabilize the the varieties $A_f$, and therefore, stabilize $A_{R,S}$. Since the conjugates of $\mathbb{K}_1(r,s)$ generate the space of holomorphic differentials on $J_{new}(R,S)$, we are done.

   Generically, it is not guaranteed that $\Gamma$ is congruence. If we show that each $\Q$-simple factor is of $GL_2$-type, by Ribet's \cite{ribet} result, this follows immediately. We can also sidestep this as follows. If $\Gamma$ is noncongruence it has a congruence closure $\Gamma(N')$ for some $N$, and a result of Berger \cite{berger} decomposes $$M_k(\Gamma)=M_k(\Gamma(N))\oplus C,$$ where $T_pC=0$ for all $p$ and acts as usual on $M_k(\Gamma(N))$. The question is then identifying if a given differential corresponds to an element of $C$ or of $M_k(\Gamma)$. But in our case, the $\mathbb{K}_1(r,s)$ functions for $(r,s)\in \mathbb S_1$ are by definition modular forms on $M_k(\Gamma(48))$, which implies $J_{new}(R,S)$ must still be isomorphic to $A_f$ for some congruence newform $f$, rather than $A_f$ for some $f\in \C$. Then, we can apply the same proof as above.
\end{proof}

From here, the remainder of the proof of Theorem \ref{cons} is extremely similar to the proof given by the author in \cite{rosen}. The key Lemma is the following.

\begin{lemma}\label{hec}
    Assume $\mathbb{K}_1(a_j/b,s_j/m)$ is congruence and holomorphic, and $a_1,...,a_n$ are the residues mod $b$. Also take $p\equiv a_i\mod b$ to be a prime, and $a_k\equiv a_ja_i\mod b$. Then $$T_{p}\mathbb{K}_1(a_j/b,s_j/m)=C(p,i) \mathbb{K}_1(a_k/b,s_k/m)$$ for some integer $C(p,i)$. Moreover, $C(p,i)=0$ if and only if there is no pair $(a_k/b,s/m)$ conjugate to $(a_j/b,s_j/m)$ so that $\mathbb{K}_1(a_k/b,s_k/m)$ is holomorphic.
\end{lemma}
\begin{proof}
   The proof is virtually identical to the proof of Lemma 3.2 in \cite{rosen}. We omit the details. 
\end{proof}
The remainder of the proof of Theorem \ref{cons} is obtained by diagonalizing the matrix obtained via Lemma \ref{hec} above. In fact, the proof of Lemma 3.3 in \cite{rosen} carries through almost word for word, and so the $\beta_i$ are (individually) quadratic. We refer the reader to \cite{rosen} for the details. The $\beta_i$ can also easily be identified once the Hecke orbit of $\mathbb{K}_1(r,s)$ is determined using the $L$-functions and Modular Forms Database (LMFDB).

\section{Galois Representations}\label{jac}
Recall a Jacobi sum is defined as as a certain special character sum and is the \'Etale realization of a Jacobi motive. To this end, we define the Jacobi sum in a way that corresponds to the de Rham realization of the same motive we have already discussed, namely $B(r,s-r-1/2)$. Take $r$ and $s$ to be rational numbers, $M_{r,s}=M=\text{lcd}(r,s,1/2)$ as above, and let $\mathfrak{p}$ denote a prime ideal of the cyclotomic field coprime to $M$, and let $p$ denote the rational prime below $\mathfrak{p}$. Then there is a character $\omega$ mod $p$ of order $M$ so that $\omega(x) \equiv x^{\frac{p-1}{M}} \mod \mathfrak{p} \quad \forall x\in \Z[\zeta_M]$, called the $M$th residue symbol. Under these assumptions, we define the map $\iota(i/M)=\omega^i$. This notation is borrowed from \cite{HMM1}; see that paper for more discussion, as well as Weil's original paper \cite{weil}. We define \begin{equation}\label{jacobi}
     J_\mathfrak{p}(r,s)=\sum_{k=1}^{p-1}\iota(r)(k)\iota(s)(1-k).
 \end{equation}   Weil \cite{weil} showed that $ -J_\mathfrak{p}(r,s)$ can be identified with a specific Hecke Gr\"ossencharakter of $\Q(\zeta_M)$ modulo $M^2$ when we allow $\mathfrak{p}$ to vary. This is the \'Etale realization of the Jacobi motives mentioned above. This is precisely why Jacobi sums are said to \say{correspond} to beta values in the so-called dictionaries between the complex and finite field setting introduced by Greene \cite{greene}.
 
In \cite{HMM2} by Allen et al., Lemma 4.4 provides an explicit modularity result for $\mathbb{K}_1(1/D,3/2-1/D)$, where $D=3,4,8,12,24$ in terms of Jacobi sums. As expected, this relates the representation attached to the newform $f_D$ associated to $\mathbb{K}_1(1/D,3/2-1/D)$ with a direct sum of Gr\"ossencharakters attached to Jacobi sums. Let $p\equiv 1\mod M$ be a prime and $\mathfrak{p}$ a prime ideal above $p$ in $\Q(\zeta_M)$. Since $M=\text{lcd}(r,s,1/2)$, $M=D$ if $D$ is even and $M=2D$ otherwise. Explicitly, they showed $$a_p(f_D)=-[J_{\mathfrak{p}}(1/D,1-2/D)+J_{\mathfrak{p}}(1/D+1/2,-2/D)]\omega_p(2)^{-4(p-1)/D},$$ where $J_{\mathfrak{p}}(r,s)$ is defined as in equation \eqref{jacobi}. Here $\omega_p$ is a  generator of the character group of $(\Z/p\Z)^\times$. A power of this character is equal to $\chi_{\mathfrak{p}}$ used in the definition of the Jacobi sum in equation \eqref{jacobi}, as $M$ divides $p-1$, and $p\equiv 1\mod M$. Therefore, the Jacobi sums can be implemented using only $\omega_p$. We fix an embedding into $\Q_p$ using the so-called Teichm\"uller character. 
The beta values corresponding to these two Jacobi sums are $B(1/D,1-2/D)$ and $B(1/D+1/2,-2/D)$ correspond to periods of a differential of the first kind and a non-holomorphic differential on $J_{new}(1/D,1-2/D)$. This suggests that the two characters that the Galois representation attached to $f_D$ decomposes into correspond to a differential of the first kind and, essentially, a differential of the second kind.

In this section, we will generalize Lemma 4.4 of \cite{HMM2} to any congruence $\mathbb{K}_1(r,s)$ function. The proof strategy will be similar.

The first step is to prove the congruence. Following \cite{HMM2}, this can be achieved using commutative formal group laws. The calculation below is follows from the work of the same authors in \cite{HMM1}, but we sketch the details here for completeness. 

\begin{lemma}
  For $p\equiv 1\mod M$ and $f_{r,s}$ the Hecke eigenform associated to the pair $(r,s)$ in Table \ref{tab:K1}, $$-\omega_p^{-4(p-1)r}(2) \cdot J_{\mathfrak{p}}(r,s-r-1/2)\equiv a_p(f_{r,s})\mod p.$$
\end{lemma}

\begin{proof}
   We compute $\mathbb{K}_1(r,s)$ as a power series in two different ways, and show that the underlying commutative formal group laws are isomorphic, and so the $p$th coefficient of each power series are congruent to each other modulo $p$, (see \cite{grouplaw}). The first power series arises from expanding $\lambda^r(1-\lambda)^{s-r-3/2}$ in terms of $\lambda^r$ by using the binomial series for $(1-\lambda)^{s-r-3/2}$. The series is \begin{equation*}\label{binomial}
    (1-\lambda)^{s-r-3/2}=\sum_{k=0}^\infty \frac{(s-r-3/2)_k}{k!}\lambda^k.
\end{equation*}
Following \cite{HMM1}, we define $A_{r,s}(k):=\frac{(s-r-3/2)_k}{k!}$. Assume $r=i/M$, where $M$ is the lowest common denominator of $1/2,r,s$, and $i$ is an integer, not necessarily coprime to $M$. Take $b={M}/{\text{gcd}(i,M)}. $ We will make the substitution $u=\lambda^{1/b}$, which produces the following expansion for $\mathbb{K}_1(r,s)$.
\begin{equation}\label{uexp}
    \mathbb{K}_1(r,s)=\frac{M}{\text{gcd}(i,M)}\sum_{k=0}^\infty A_{r,s}(k)u^{kb+a}\frac{du}{u}.
\end{equation}
This expansion is exactly what we need, since by an application of the Gross-Koblitz formula \cite{grosskoblitz}, in \cite{HMM1} they show that \begin{equation}
    A_{r,s}((p-1)r)\equiv -J_{\mathfrak{p}}(r,s-r-1/2)\mod p,
\end{equation} assuming $r$, $s-r-1/2$, and $s-1/2$ are all in the interval $(0,1)$, which is always true for a congruence form by the definition of $\mathbb{S}_1$.
The extra 1/2 compared to the equation in \cite{HMM2} is because the exponent $(1-\lambda)$ is $s-r-1$, while we use $s-r-3/2$. However, the exact same argument still holds otherwise.

We can also expand $\mathbb{K}_1(r,s)$ using the Fourier expansion in terms of $s_b:=e^{2\pi i \tau/b}$. We do this by replacing $\lambda$ with its Fourier expansion, $\lambda=16s+O(q^2)$. We may therefore write $$b\mathbb{K}_1(r,s)=\sum_{k=0}^\infty c_k u^k=16^r\sum_{k=1}^\infty b_k q_b^k.$$ To obtain the coefficient $A_{r,s}(r(p-1))$, we use $c_{mp}$ in the series expansion above, and the $p$th Fourier coefficient $a_p(f_{r,s})$ mod $p$ is $b_p$. Using the properties of the formal group law (see \cite{HMM1}, Equation 3.20), the authors of \cite{HMM1} show that the relationship between $c_{mp}$ and $b_p$ is $$\iota_{\mathfrak{p}}(-i/M) \cdot c_{p\ell}\equiv b_p\cdot c_\ell \mod p,$$ where $\ell$ is some integer. Recall $\iota_{\mathfrak{p}}(-i/M)=\chi_{\mathfrak{p}}^{-i}.$ Taking $\ell=1$ and note that $c_1=1$ when expanding $\mathbb{K}_1(r,s)$, we get that $\chi_{\mathfrak{p}}(16)^{-i} \cdot c_{p}\equiv b_p$. We also have that $\chi_\mathfrak{p}(16)^{-i}=\omega_p^{-4(p-1)r}(2)$. By construction, $b_p\equiv a_p(f_{r,s})\mod p$ and $$c_p\equiv A_{r,s}(r(p-1))\equiv -J_{\mathfrak{p}}(r,1-r-1/2)\mod p,$$ and so we deduce the desired congruence.
\end{proof}

\begin{theorem}\label{gal}
    For $p\equiv 1\mod M$, we have $$a_p(f_{r,s})= -\omega_p^{-4(p-1)r}(2)\cdot  J_{\mathfrak{p}}(r,s-r-1/2)-\omega_p^{-4(p-1)(1-r)}(2) \cdot J_{\mathfrak{p}}(1-r,1/2-s+r).$$
\end{theorem}

\begin{proof}
 The Gross-Koblitz formula (see \cite{grosskoblitz}) applied to this shows that $$-J_{\mathfrak{p}}(1-r,1/2-s+r)\equiv \pi_p^{p-1}\frac{\Gamma_p(1-r)\Gamma_p(3/2-s+r)}{\Gamma_p(3/2-s)}\mod p.$$ This occurs because $1/2-s+r$ is always less than zero, and so we have to correct it to $3/2-s+r$ when using the Gross-Koblitz formula. However, the term $1/2-s+r+1-3=3/2-s$ is always between one and zero, and so no correction is needed here. This means an extra $\pi_p^{p-1}$ appears compared to Lemma 3.4 of \cite{HMM1}. Since $\pi_p^{p-1}=-p$ by definition, we have $$-\omega_p^{-4(p-1)(1-r)}(2) \cdot J_{\mathfrak{p}}(1-r,1/2-s+r)\equiv 0\mod p.$$ Therefore $$a_p(f_{r,s})\equiv -\omega_p^{-4(p-1)r}(2) \cdot J_{\mathfrak{p}}(r,s-r-1/2)-\omega_p^{-4(p-1)(1-r)}(2) \cdot J_{\mathfrak{p}}(1-r,1/2-s+r)\mod p.$$ We claim this is an equality. 

First, we show that the right hand side of the congruence above is an integer. The key ingredient is that the set $\{J_{\mathfrak{p}}(r_i,s_i-r_i-1/2)\}$ are all each others Galois conjugates. Note each family of conjugates can be written as $$\{\mathbb{K}_1(i/b,s_i)\}_{i\in I_{1/b,s_1}}.$$ If $\sigma_j(z)$ is as above, $j\in I_{1/b,s_1}$, and $ij\equiv k\mod b$, then in particular we have 
\begin{equation}\label{switch}
    \sigma_j(J_{\mathfrak{p}}(i/b,s_i-i/b-1/2))=J_{\mathfrak{p}}(k/b,s_k-k/b-1/2).
\end{equation} This is because each Jacobi sum corresponds the trace of Frobenius for a representation corresponding to an element of the \'Etale cohomology, while each period is a realization of the de Rham cohomology of the same motive. That is, the Jacobi sum $J_{\mathfrak{p}}(r_i,s_i,r_i-1/2)$ corresponds to a distinct differential of the first kind with period a multiple of $B(r_i,s_i,r_i-1/2)$. There is a similar relation between differentials of the second kind and the \say{dual}, $J_{\mathfrak{p}}(1-r,1/2-s+r)$. But, we know the relation between the periods is expressed by the period matrix, and the columns of the period matrix $L_{R,Q}$ for $J_{new}$ are $$(\sigma_j(1),\sigma_j(\zeta_M),....,\sigma_j(\zeta_M^{M-1}))^\top,$$ with each $j\in I_{1/b,s_1}^\times$ arising from a differential of the first kind. A similar property for the remaining $j$ holds for differentials of the second kind. This suggests the $\sigma_j$ switch the Jacobi sums corresponding to each differential of the first or second kind for motivic reasons. Because there is not an explicit formula for $I_{r,s}$, there is not an easy way to show that the automorphisms $\sigma_j$ switch as in equation (\ref{switch}) explicitly as in \cite{HMM2}, but we can check this case by case as well. 

The remainder of the proof is similar to \cite{HMM2}. First, set $$v_{r,s}= -\omega_p^{-4(p-1)r}(2) \cdot J_{\mathfrak{p}}(r,s-r-1/2)-\omega_p^{-4(p-1)(1-r)}(2)\cdot  J_{\mathfrak{p}}(1-r,1/2-s+r)-a_p(f_{r,s}).$$ Let $u_1,...,u_m$ be an integral basis of $\Z[\zeta_M]$; then there are integers $b_1,...,b_m$ so that $$v_{r,s}=\sum_{i=1}^m b_i u_i.$$ By above, we know $v_{r,s}\equiv 0\mod p$. From this discussion above, each automorphism $\sigma_j$ maps to a conjugate or its dual and fixes the integer $a_p(f_{r,s})$, and so we have $\sigma_j(v_{r,s})\equiv 0\mod p$ for each $j\in (\Z/M\Z)^\times$. Therefore, $b_i\equiv 0\mod p$ for each $p$ as well. We also have the bounds $|J_{\mathfrak{p}}(a,b)|=\sqrt{p}$ and $|a_p(f_{r,s})|<2\sqrt{p}$ by the standard properties of Jacobi sums (see \cite{irelandrosen}) and the Deligne bound on newforms. Therefore, we know $$|v(r,s)|<4\sqrt{p}.$$ Since $p\mid b_i$ for all $i$, this implies $b_i=0$ for all $i$ when $p>17$ and so $v_{r,s}=0$. This proves the desired equality. The remaining smaller $p$ can be checked by hand.
\end{proof}

\section{Special \textit{L}-values}
By the method of \cite{HMM2}, we use Theorem \ref{cons} to obtain exact $L$-values of the weight 2 CM newforms. 
\begin{corollary}\label{lvale}
    Let $f$ be the newform given in Theorem \ref{cons}. Then $$L(f,1)=\beta_1L(r_1,q_1)+\beta_2 L(r_1,q_1)+...+\beta_n L(r_n,q_n).$$
\end{corollary} 
The proof follows immediately from integrating each side of Theorem \ref{cons}. Actually, these $L$-values can be simplified on a case by case basis to give much cleaner expression. In particular, $L(r_i,s_i)$ and $L(r_j,q_j)$ are an algebraic multiple of each other.
The following Lemma is a corollary of Theorem 1 in \cite{shimura2} by Shimura.

\begin{lemma}\label{chowl}
    For $(r,s)\in \mathbb{S}_1$, we have $B(r,s-r-1/2)\in \Omega_{-D}\cdot \bar{\Q}$ for some fundamental discriminant $-D$.
\end{lemma}
\begin{proof}
   We know from above that $B(r,s-r-1/2)$ is a period $J_{new}(R,S)$, and that this variety is a factor of the Jacobian for a modular curve. We also know that $J_{new}(R,S)$ has CM. Then by Shimura's Theorem 1.6 \cite{shimura2},$J_{new}(R,S)$ is isogenous to $E_{-d}^{
   \varphi(M)}$ over $\ol \Q$, where $E_{-d}$ is an elliptic curve defined over $\ol \Q$ with CM by $\Q(\sqrt{-d})$. Moreover, each of the $\mathbb K_1(r,s)$ in the given Hecke orbit gives a holomorphic differential $1$-form with period $B(r,s-r-1/2)$. There exists a positive integer $n$ such that the path 
$n\{ 0,\infty\} \in  H_1(X_1(N),\Z)$ (see \cite{stein}) and integrating over this path produces a period of the elliptic curve. 
Hence, $ B(r,s-r-1/2)$ gives a non-zero period of the elliptic curve $E_{-d}$ up to an algebraic number, and so is an algebraic multiple of a Chowla-Selberg period. 
\end{proof}

\begin{remark}
    Because each $J_{new}(R,S)$ is isogenous to a power of a CM elliptic curve, if $B$ is a period of a differential of the first kind and $B'$ is a period of a non-holomorphic differential, then $B\cdot B'$ is an algebraic multiple of $\pi$. The choice of non-holomorphic differential is non-canonical, but this property holds regardless of the choice. For example, if we choose $\mathbb{K}_1(1-r,2-r)(\tau)d\tau$ as in Theorem \ref{gal}, the period is $B(1-r,1/2-s+r)$, and we have the relation $$B(r,s-r-1/2)B(1-r,1/2-s+r)=\frac{2\pi \cos(\pi q)\csc(\pi r)\sec(\pi(s-r))}{1-2s+2r}$$ using the properties of the $\Gamma$ function (see \cite{aar}). 

\end{remark}
\subsection{Computing \textit{L}-values}
The elliptic curve appearing in the decomposition of $A_f$ is naturally defined over $K_f$, the Hecke eigenvalue field of $f$. When $K_f=\Q$, this produces the $L$-value of an elliptic curve defined over $\Q$, which are well understood. When $K_f$ is larger than $\Q$, we take $n_f:=[K_f:\Q]$. Then there are $n_f$ inner twists of $f$, which are automorphisms of $K_f$ that act on the Fourier coefficients of $f$. Each inner twist has a different $L$-value at 1, and using our method we can compute the $L$-value for any inner twist by applying inner twist $\sigma$to the $\beta_i$ given in Theorem \ref{cons}. Let $\sigma_1,...,\sigma_{n_f}$ denote the automorphisms of $K_f$. We denote the $L$-value for the inner twist $\sigma_i$ by $$L(f^{\sigma_i},1).$$ If $A_f$ is the Abelian variety attached to $f$, we have $$L(f,1):=L(A_f,1)=\prod_{i=1}^{n_f} L(f^{\sigma_i},1);$$ we use $L(f,1)$ only to denote the $L$-value corresponding to the product of the inner twists in this paper, which differs from the convention in \cite{rosen}. By Lemma \ref{chowl} above, we have $L(f^{\sigma_i},1)=D_{f^{\sigma_i}}\cdot \Omega_{-d},$ where $D_{f^{\sigma_i}}\in \Bar{\Q}$. Multiplying these, it follows that $L(f,1)=D_f\Omega_{-d}^{n_f}$ for an algebraic number $D_f$. All of this is due to the classical theory of CM Abelian varieties. Our method is that it allows us to quickly prove $D_f$ the exact value of as well.  We illustrate with an example over a number field below.

Consider the family $$\{\mathbb{K}_1(1/24,17/24), \mathbb{K}_1(7/24,23/24), \mathbb{K}_1(13/24,29/24), \mathbb{K}_1(19/24, 35/24)\}.$$ The corresponding newform is $f_{576.2.d.c}$, and Hecke eigenvalue field is $\Q(\zeta_{12})$. The product of the four conjugates produces the value
$$L({f_{576.2.d.c}},1)=\frac{1}{216\sqrt[3]{2}}\Omega_{-3}^4.$$ This is the same as the $L$-value of the Abelian variety attached to $f$. We can obtain 32 exact $L$-values using this method, including quadratic twists that are not inner twists.

\subsection{Twisting and Relations between \textit{L}-values}

 In \cite{rosen}, the author finds a twisting theorem for the weight 3 case that that is connected to a classical transformation due to Kummer (see \cite{aar} Corollary 3.3.5) and uses this to find relation among $L$-values of weight 3 non-CM modular forms. A similar twisting result allows us to sort our $L$-values in terms of twists of a single $L$-function. The proof is similar to the $\mathbb{K}_2(r,s)$.
\begin{proposition}\label{twisting3}
        The Fourier coefficients $A_n$ of $\mathbb{K}_1(r,s)$ and $B_n$ of $\mathbb{K}_1(r,r-s+2)$ differ by a sign when both are holomorphic, that is $B_n=\pm A_n$ for all $n$.

\end{proposition}

\begin{proof} Recall that under the translation $\tau \mapsto \tau+1$, $\l(\tau+1)=\frac{\l(\tau)}{\l(\tau)-1}$.
    Make the change of variable $t=\lambda/(\lambda-1)$  on $\lambda^{r}(1-\lambda)^{s-r-3/2}\frac{d\lambda}{\lambda}$, and using this compute that $$\mathbb{K}_1(r,s)(\tau)=(-1)^r\mathbb{K}_1(r,r-s+2)(\tau+1).$$ This implies the Fourier expansions differ by a root of unity, and since the $\mathbb{K}_1$ functions have integral Fourier coefficients, they differ by a sign.
\end{proof}
This result extends to show that $f_{r,s}$ and $f_{r,r-s+2}$ differ by a twist of a character for the $f_{r,s}$ listed in Table \ref{tab:K1}. These twists are recorded as \say{other twists} in Section \ref{lvalues}. 

\subsection{Tables of \textit{L}-values}\label{lvalues}
In this section, we list all of the $L$-values that we can obtain using our construction, besides quadratic twists. The $L$-values are for the Abelian variety $A_f$ attached to the modular form, that is we take the product of the conjugates. We also indicate the varieties which are related by a twist by Proposition \ref{twisting3} above.

First, we have the modular forms with CM discriminant $-4$ by $\Q(\sqrt{-1})$. There are 15 such forms, and 4 up to twisting.

\smallskip

\begin{center}
    {\renewcommand{\arraystretch}{1.3}%
 \begin{tabular}{|c|c|c|c|c|}\hline
        Label & LMFDB & $L$-value & Quadratic twists & Other twists \\\hline
        4.1 & 32.2.a.a & $\Omega_{-4}$ & N/A & 4.2 \\
        4.2 & 64.2.a.a & $\Omega_{-4}/\sqrt{32}$ & N/A & 4.1 \\
        4.3 & 128.2.b.b & $\Omega_{-4}^2/\sqrt{512}$ & N/A & 4.4 \\
        4.4 & 256.2.a.c & ${\Omega_{-4}}/{\sqrt[4]{128}}$ & 256.2.a.b & 4.3 \\
       4.5  & 288.2.a.e & $\Omega_{-4}/\sqrt[4]{108}$ & 288.2.a.a & 4.6\\
        4.6 & 576.2.a.a & $\Omega_{-4}/\sqrt[4]{432}$ & 576.2.a.i & 4.5\\
        4.7 & 1152.2.d.b & $\Omega_{-4}^2/\sqrt{864}$ & 1152.2.d.e & 4.8\\\hdashline
       4.8 & 2304.2.a.d & $\Omega_{-4}/\sqrt[4]{216}$ & $\begin{array}{c}
             2304.2.\text{a.c}  \\
             2304.2.\text{a.m}\\ 
             2304.2.\text{a.n}
        \end{array}$ &  4.7\\\hline 
                        \end{tabular}}

\smallskip
\end{center}

\noindent There are 7 examples with CM discriminant $-3$ by $\Q(\sqrt{-3})$, and 3 up to twisting.

\smallskip
\begin{center}
{\renewcommand{\arraystretch}{1.3}
 \begin{tabular}{|c|c|c|c|c|}\hline Label & LMFDB & $L$-value & Quadratic twists & Other twists \\\hline
        3.1 & 36.2.a.a & $\Omega_{-3}/(2^{5/6}\cdot 3^{5/4})$ & N/A & 3.2, 3.3 \\
         3.2  & 144.2.a.a  & $\Omega_{-3}/(2^{5/6}\cdot 3^{3/4})$ & N/A &  3.1, 3.3\\
           3.3  & 576.2.a.f & $\Omega_{-3}/(2^{1/3}\cdot 3^{3/4})$ & 576.2.a.e & 3.1, 3.2 \\
             3.4  & 576.2.d.c & $\Omega_{-3}^4/(216\cdot 2^{1/3})$ & N/A & 3.5 \\
              3.5   & 2304.2.a.w & $\Omega_{-3}^2/(3^{3/2}\cdot 2^{2/3})$ & 2304.2.a.v & 3.4\\\hline
                       \end{tabular}}
\end{center}
                       \smallskip
                       For discriminant $-8$ by $\Q(\sqrt{-2})$, there are 5 cases, and 2 up to twisting. 
                       \smallskip
          \begin{center}   {\renewcommand{\arraystretch}{1.3}          
                    \begin{tabular}{|c|c|c|c|c|}\hline
        Label & LMFDB & $L$-value & Quadratic twists & Other twists \\\hline
               8.1    & 64.2.b.a & $\Omega_{-8}^2/\sqrt{2048}$ & N/A & 8.3 \\
                  8.2   &  128.2.b.a & $\Omega_{-8}^2/32$ & N/A & 8.4\\
                   8.3    & 256.2.a.d & $\Omega_{-8}/\sqrt[4]{128}$ & 256.2.a.a & 8.1\\
                  8.4 & 256.2.a.e & $\Omega_{-8}^2/\sqrt{512}$ & N/A & 8.2\\\hline
                      \end{tabular} }   
                       \end{center}
                    
                    \smallskip
                \noindent Finally, for discriminant $-24$ by $\Q(\sqrt{-6})$, we also have five examples, and 2 up to twisting.

                    \smallskip
                     \begin{center} {\renewcommand{\arraystretch}{1.3}
                 \begin{tabular}{|c|c|c|c|c|}\hline
        Label & LMFDB & $L$-value & Quadratic twists & Other twists \\\hline
                  24.1 & 288.2.d.a  & $\Omega_{-24}^2/\sqrt{1728}$ & N/A & 24.4  \\
                   24.2   & 1152.2.d.g  & $\Omega_{-24}^4/864$ & N/A &  24.3\\
                     24.3  & 2304.2.a.z  & $\Omega_{-24}^4/1728$ & N/A & 24.2 \\
                      24.4    & 2304.2.a.y & $\Omega_{-24}^2/\sqrt{216}$ & 2304.2.a.q &  24.1\\ \hline                    \end{tabular} }      
                       \end{center} 

\section{Appendix: Construction of the Hecke Eigenforms using Eta Products}

{\renewcommand{\arraystretch}{1.1}
\begin{table}[H]
    \centering
    \begin{tabular}{|c|c|c|c|c|}\hline 
 Label &     LMFDB & Hecke Eigenform  &CM \\\hline
4.1& 32.2.a.a & $\mathbb{K}_1(1/4, 5/4) $ &$-4$\\\hline
4.2& 64.2.a.a & $\mathbb{K}_1(1/4, 1)$&$-4$\\\hline
4.3& 128.2.b.b &    $\mathbb{K}_1(1/8, 7/8) + 4 i \mathbb{K}_1(5/8, 11/8)$ & $-4$ \\\hline
4.4& 256.2.a.b &   $\mathbb{K}_1(1/8, 5/4) - 4 \mathbb{K}_1(5/8, 5/4)$ & $-4$\\\hline
4.5& 288.2.a.a  &   $\mathbb{K}_1(1/6, 13/12)(12\tau) - 4 \mathbb{K}_1(5/6, 17/12)(12\tau)$  
             &$-4$\\\hline
4.6& 576.2.a.a & $\mathbb{K}_1(1/12, 2/3) - 4 \mathbb{K}_1(5/12, 4/3)$ &$-4$\\\hline 
4.7& 1152.2.d.b  & $\mathbb{K}_1(1/24, 23/24) + 2 i \mathbb{K}_1(5/24, 19/24)$&\\&&$ + 4 i \mathbb{K}_1(13/24, 35/24) - 8 \mathbb{K}_1(17/24, 31/24)$  &$-4$\\\hline
4.8 & 2304.2.a.c  & $\mathbb{K}_1(1/24, 13/12) - 2  \mathbb{K}_1(5/24, 17/12) $ &\\&&$ - 4  \mathbb{K}_1(13/24, 13/12)+ 
 8 \mathbb{K}_1(17/24, 17/12)$ & $-4$\\\hline

3.1 & 36.2.a.a   &   $\mathbb{K}_1(1/3,7/6)(6\tau)$  & $-3$\\\hline
3.2 & 144.2.a.a & $\mathbb{K}_2(1/6, 5/6)$ &$-3$\\\hline
3.3& 576.2.a.f & $\mathbb{K}_1(1/12, 7/6) + 4 \mathbb{K}_1(7/12, 7/6)$ &$-3$\\\hline
3.4& 576.2.d.c & $\mathbb{K}_1(1/24, 17/24) - 2 \sqrt{3} \mathbb{K}_1(7/24, 23/24) $&\\&&$- 
 4 \sqrt{3}i \mathbb{K}_1(13/24, 29/24) - 8 i\mathbb{K}_1(19/24, 35/24) $ &$-3$\\\hline
3.5 & 2304.2.a.v & $\mathbb{K}_1(1/24, 4/3) + 2 \sqrt{3} \mathbb{K}_1(7/24, 4/3) $&\\&&$- 4 \sqrt{3} \mathbb{K}_1(13/24, 4/3) - 
 8 \mathbb{K}_1(19/24, 4/3)$&$-3$\\\hline
8.1& 64.2.b.a   &   $      \mathbb{K}_1(1/8, 9/8) + 2 i\mathbb{K}_1(3/8, 11/8)$ &$-8$ \\\hline
8.2 & 128.2.b.a  &     $\mathbb{K}_1(1/8, 11/8) + 2 \sqrt{2}i \mathbb{K}_1(3/8, 9/8)$ & $-8$ \\\hline
8.3& 256.2.a.d    &  $\mathbb{K}_1(1/8, 1) + 2 \mathbb{K}_1(3/8, 1)$   &$-8$ \\\hline
8.4& 256.2.a.e  &    $\mathbb{K}_1(1/8, 3/4) + 2 \sqrt{2}\mathbb{K}_1(3/8, 5/4)$ &$-8$\\\hline
 
24.1&288.2.d.a   &    $\mathbb{K}_1(1/24, 5/6) + 2 \sqrt{2} \mathbb{K}_1(5/24, 7/6) $&\\&&$ - 2 \mathbb{K}_1(7/24, 5/6) - 
 4 \sqrt{2} \mathbb{K}_1(11/24, 7/6)$ $-24$ &$-24$\\\hline
24.2&1152.2.d.g &    $\mathbb{K}_1(1/12, 25/24) + 2 \sqrt{3}i\mathbb{K}_1(5/12, 29/24)$ & \\&& $+ 2 \sqrt{6} \mathbb{K}_1(7/12, 31/24) + 4 \sqrt{2}i \mathbb{K}_1(11/12, 35/24)$  &$-24$ \\\hline
24.3&2304.2.a.z    &   $\mathbb{K}_1(1/24, 7/12) + 2 \sqrt{3} \mathbb{K}_1(5/24, 11/12)$ & \\&& $- 2 \sqrt{6} \mathbb{K}_1(7/24, 13/12) + 4 \sqrt{2} \mathbb{K}_1(11/24, 17/12)$   &$-24$ \\\hline
24.4&2304.2.a.q   &  $\mathbb{K}_1(1/24, 29/24)+ 2 \sqrt{2}i \mathbb{K}_1(5/24, 25/24)$ &\\&& $- 2 \mathbb{K}_1(7/24, 35/24) + 4 \sqrt{2}i\mathbb{K}_1(11/24, 31/24)$ &$-24$ \\\hline

    \end{tabular}
    \caption{Completing $\mathbb{K}_1$ functions to Hecke eigenforms}
    \label{tab:K1}
\end{table}}

\bibliographystyle{plain} 
\bibliography{refs}{}

\end{document}